\documentclass{amsart}

\usepackage{amsmath,amssymb}
\usepackage{url}
\usepackage{color}
\usepackage{verbatim}
\usepackage{standalone}
\usepackage{tikz}
\usepackage{verbatim}
\usetikzlibrary{arrows}
\usepackage{algorithm}
\usepackage[noend]{algpseudocode}

\newtheorem{thm}{Theorem}

\newtheorem{lem}[thm]{Lemma}
\newtheorem{cor}[thm]{Corollary}
\newtheorem{conj}[thm]{Conjecture}

\newtheorem{rem}[thm]{Remark}

\newcommand{\C}{\mathcal{C}}
\newcommand{\F}{\mathcal{F}}
\newcommand{\G}{\mathcal{G}}
\newcommand{\N}{\mathcal{N}}
\renewcommand{\P}{\mathcal{P}}
\newcommand{\X}{\mathcal{X}}

\newcommand{\bC}{\mathbf{C}}
\newcommand{\bF}{\mathbf{F}}
\newcommand{\bP}{\mathbf{P}}

\newcommand{\R}{\mathbb{R}}

\newcommand{\pathto}{\rightsquigarrow}

\newcommand{\Gr}{\mathop{\rm Gr}}
\renewcommand{\int}{\mathop{\rm int}}
\newcommand{\inv}{\mathop{\rm inv}}
\newcommand{\sgn}{\mathop{\rm sgn}}

\newcommand{\wt}{\mathop{\rm wt}}
\newcommand{\xing}{\mathop{\rm xing}}

\begin{document}

\title{Boundary Measurement Matrices for Directed Networks on Surfaces}
\author{John Machacek}
\address{Department of Mathematics, Michigan State University, USA}
\email{machace5@math.msu.edu}


\begin{abstract}
Franco, Galloni, Penante, and Wen have proposed a boundary measurement map for a graph on any closed orientable surface with boundary.
We consider this boundary measurement map which takes as input an edge weighted directed graph embedded on a surface and produces on element of a Grassmannian.
Computing the boundary measurement requires a choice of fundamental domain.
Here the boundary measurement map is shown to be independent of the choice of fundamental domain
Also, a formula for the Pl{\"u}cker coordinates of the element of the Grassmannian in the image of the boundary measurement map is given.
The formula expresses the Pl{\"u}cker coordinates as a rational function which can be combinatorially described in terms of paths and cycles in the directed graph.
\end{abstract}

\maketitle

\section{Introduction}
The totally nonnegative Grassmannian was defined by Postnikov~\cite{Pos06} and can be studied using edge weighted planar graphs embedded on a disk.
These edge weighted planar graphs and the totally nonnegative Grassmannian are connected to the physics of scattering amplitudes and $\N = 4$ super Yang-Mills~\cite{AHBC12}.
In the context of physics, the edge weighted planar graphs are usually called ``on-shell diagrams."
A key element of Postnikov's study of the totally nonnegative Grassmannian is the boundary measurement map which produces an element of the totally nonnegative Grassmannian for any edge weighted directed graph embedded in the disk.
Under a mild hypothesis on the graph, Talaska~\cite{Tal08} gives a formula for the Pl{\"u}cker coordinates of the element of the totally nonnegative Grassmannian corresponding to a given graph.
In~\cite{FGM14, FGPW15} a boundary measurement map for graphs on more general surfaces is proposed with the hopes of going beyond the ``planar limit'' of $\N = 4$ super Yang-Mills.

The definition of the boundary measurement map will be given later in this section, and in defining the boundary measurement we must make a choice of how to represent our directed graph in the plane.
The boundary measurement map turns out to be independent of this choice as we will see in Section~\ref{sec:ind}.
We will show in Section~\ref{sec:sign} how boundary measurement map can be obtained by signing the edges of a directed graph.
This technique of signing edges will allow us to unify two formulas of Talaska~\cite{Tal08, Tal12}. 
A formula for the Pl{\"u}cker coordinates corresponding to the boundary measurement map is given in Section~\ref{sec:Plucker}.
In Section~\ref{sec:gauge} we will show that the signs used in Section~\ref{sec:sign} are unique up to the gauge action.

\subsection{Weighted Path Matrices}
Let $N = (V,E)$ be a directed graph with finite vertex set $V$ and finite edge set $E$.
This means an edge $e \in E$ is an ordered pair $e = (i,j)$ for $i,j \in V$.
If $e = (i,j)$ then the edge $e$ is said to be directed from vertex $i$ to vertex $j$.
For each edge $e \in E$ of $N$ we associate a formal variable $x_e$.
We will work in $\R[[x_e : e \in E]]$ the ring of formal power series in the variables $\{x_e\}_{e \in E}$ with coefficients in $\R$.
As in~\cite{Pos06}, we will use the term $\emph{directed network}$ $N = (V,E)$ to refer to the directed graph $N = (V,E)$ along with edge weights $\{x_e\}_{e \in E}$.

A \emph{path} is a finite sequence of edges $P = (e_1, e_2, \cdots, e_l)$ where $e_k = (i_{k-1}, i_k)$ for $1 \leq k \leq l$.
If $P = (e_1, e_2, \cdots, e_l)$ where $e_1 = (i_0, i_1)$ and $e_l = (i_{l-1}, i_l)$, then $P$ is said to be a path from $i_0$ to $i_l$.
The path $P$ is said to be \emph{self avoiding} if $i_k \ne i_{k'}$ for $k \ne k'$.
The path $P$ is called a \emph{cycle} if $i_0 = i_{l}$, and we say the cycle is a \emph{simple cycle} when $i_{k} = i_{k'}$ if and only if $k = k'$ or $\{k,k'\} = \{0,l\}$.
We use the notation $P: i \pathto j$ to denote a path from $i$ to $j$.
When $P = (e_1, e_2, \cdots, e_l)$ we let
$$\wt(P) = \prod_{i=1}^l x_{e_i}$$
denote the \emph{weight} of the path $P$.

We order our vertex set $V$ and consider the $V \times V$ \emph{weighted path matrix} $M = M(N,\{x_e\}_{e \in E})$ with entries given by
$$M_{ij} = \sum_{P: i \pathto j} \wt(P)$$
for all $(i, j) \in V \times V$.

We let $\C(N)$ denote the set of all collections $\bC$ which consist of simple cycles that are pairwise vertex disjoint.
For $\bC \in \C(N)$ we define its weight as
$$\wt(\bC) = \prod_{C \in \bC} \wt(C)$$
and its sign as $\sgn(\bC) = (-1)^{|\bC|}$ where $|\bC|$ denotes the number of cycles in the collection $\bC$.
The empty collection $\varnothing$ is in $\C(N)$ with $\wt(\varnothing) = 1$ and $\sgn(\varnothing) = 1$.
We let $S_n$ denote the symmetric group on $[n] = \{1,2,\dots, n\}$ and consider elements $\pi \in S_n$ as bijections $\pi: [n] \to [n]$.
For $\pi \in S_n$ and any $I, J \subseteq V$ with $I = \{i_1 < i_2 < \cdots < i_n\}$ and $J = \{j_1 < j_2 < \cdots < j_n\}$ we let $\P_{I,J,\pi}$ denote the set of collections $\bP = (P_1, P_2, \dots, P_n)$ such that $P_k: i_k \pathto j_{\pi(k)}$ is self avoiding for each $k \in [n]$, and $P_k$ and $P_{k'}$ are vertex disjoint whenever $k \neq k'$.
For $\bP \in \P_{I,J,\pi}$ we define its weight as
$$\wt(\bP) = \prod_{P \in \bP} \wt(P)$$
and its sign as $\sgn(\bP) = \sgn(\pi)$.
Note if $\pi(k) = k$ we can have $P_k: i_k \pathto i_k$ be the empty path $P_k$ from $i_k$ to $i_k$ consisting of no edges, and in this case $\wt(P_k) = 1$.
We then let $\F_{I,J}(N)$ denote the collection of \emph{flows} from $I$ to $J$.
A flow from $I$ to $J$ is a pair $\bF = (\bP, \bC)$ such that $\bP \in \P_{I,J,\pi}$ for some $\pi \in S_n$, $\bC \in \C(N)$, and all paths in $\bP$ and cycles in $\bC$ are pairwise vertex disjoint.
For $\bF \in \F_{I,J}(N)$ with $\bF = (\bP, \bC)$ we define its weight as $\wt(\bF) = \wt(\bP) \wt(\bC)$ and its sign as $\sgn(\bF) = \sgn(\bP) \sgn(\bC)$.

Talaska's formula~\cite{Tal12} states
\begin{equation}
\Delta_{I,J}(M) = \frac{\sum_{\bF \in \F_{I,J}(N)} \sgn(\bF)\wt(\bF)}{\sum_{\bC \in \C(N)} \sgn(\bC)\wt(\bC)}
\label{eq:Talaska1}
\end{equation}
where $\Delta_{I,J}(M)$ denotes the minor of $M$ with rows indexed by $I$ and columns indexed by $J$.
Equation~(\ref{eq:Talaska1}) generalizes the Lindstr\"{o}m-Gessel-Viennot lemma~\cite{Lin73,GV85} which only applies to directed networks without directed cycles.
Fomin also provides of generalization of the  Lindstr\"{o}m-Gessel-Viennot lemma which allows for directed cycles~\cite{Fom01} where the sum is indexed by a minimal, but infinite, collection of paths.

\subsection{Boundary Measurement Matrices}

\begin{figure}
\centering
\includegraphics[scale=.7]{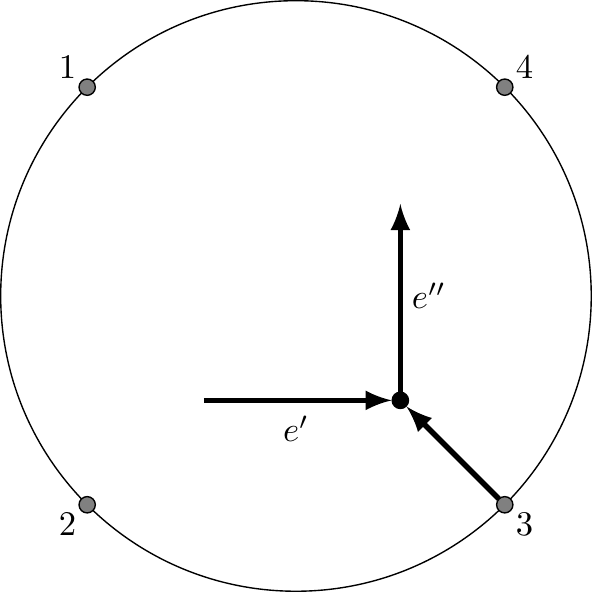}
\caption{An example of the edge order induced by the boundary vertex order. Here we say $e' < e''$ given the boundary vertices are ordered as usual with $1 < 2 < 3 < 4$. The direction of the edges is irrelevant for the induced edge ordering.}
\label{fig:edgeorder}
\end{figure}

Now consider the directed network $N = (V,E)$ embedded in a closed orientable surface with boundary $S$.
We call a vertex on the boundary of $S$ a \emph{boundary vertex} and an edge which is incident on a boundary vertex an \emph{external edge}.
We assume each boundary vertex is either a source or sink and that edges are embedded as smooth curves.
Let $K \subseteq V$ be the collection of boundary vertices.
Let $b$ denote the number of boundary components of $S$ and assume each boundary component is a smooth curve diffeomorphic to a circle.
We make $b-1$ cuts between pairs of boundary components on the surface $S$ to obtain a new surface $T$ with a single boundary component.
The cuts are made such that each cut is a smooth curve, no cut intersects any vertex of $N$, and cuts intersect edges of $N$ transversally.
The boundary $\partial T$ is then a piecewise smooth curve homeomorphic to a single circle.
We choose a piecewise smooth parameterization $\phi:[0,1] \to \partial T$ with $\phi(0) = \phi(1)$ (i.e. $\phi: S^1 \to \partial T$).
Throughout we will assume all parameterizations are piecewise smooth with nowhere zero derivative.
We order the boundary vertices so that they appear in order when traversing $\partial T$ according to $\phi$.
Thus we have a linear ordering of the vertices in $K$ which we denote by $<$ so that $K = \{i_1 < i_1 < \cdots < i_n\}$ with $i_j = \phi(t_j)$ for $0 \leq t_1 < t_2 < \cdots < t_n < 1$.
The linear ordering of the the boundary vertices induces an ordering on the set of edges incident on some external edge as demonstrated in Figure~\ref{fig:edgeorder}.
We also have a cyclic ordering which we denote $\prec$.
For $i,j,k \in K$ we write $i \prec j \prec k$ if $i < j < k$, $k < i < j$, or $j < k < i$.

\begin{figure}
\centering
\includegraphics{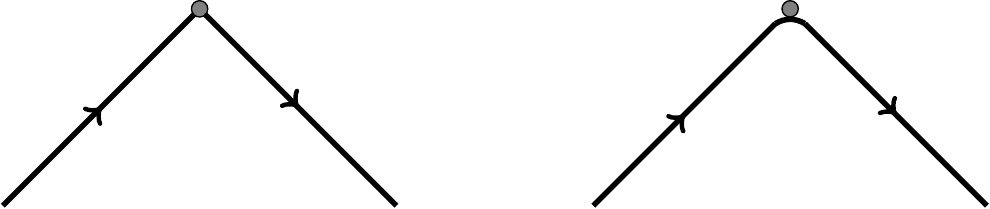}
\caption{Smoothing a piecewise smooth curve at a vertex.}
\label{fig:smoothing}
\end{figure}

When $S$ is a closed orientable surface with boundary of genus $g=0$ any network embedded on $S$ can be drawn in the plane.
In order to draw the directed network in the plane we must choose a boundary component of $S$ called external and identify this external boundary component with a circle bounding a disk in the plane.
We then draw the directed network inside this disk.
In Section~\ref{sec:ind} we will show that this choice of external boundary component does not have an impact on our results.
Consider a network $N$ on $S$ embedded in the plane and overlay the cuts used to construct $T$.
We will make use of both $S$ and $T$.
For any path $P:i \pathto j$ where $(i,j) \in I \times K$ we form a closed curve $C(P)$ in the plane as follows:
\begin{enumerate}
\item Traverse the path $P$ from $i$ to $j$ in $S$.
\item Follow the boundary of $T$ in our specified direction from $j$ to $i$.
\end{enumerate}
We want $C(P)$ to be a smooth curve.
Since we have assumed that all edges and boundary components are smooth curves the curve $C(P)$ will be piecewise smooth.
In order to work with a smooth curve we will approximate $C(P)$ by a smooth curve at cut points and around each vertex as in Figure~\ref{fig:smoothing}.
We will make no distinction between $C(P)$ and the smooth curve we approximate it by, and in some cases we may draw a piecewise smooth curve in place of a smooth curve.
Given any smooth closed curve in the plane define its \emph{rotation number} to be the degree of the map $\vec{\mathbf{T}} \circ \psi: S^1 \to S^1$ where $\psi: S^1 \to C$ is a parameterization of $C$ and $\vec{\mathbf{T}}: C \to S^1$ gives the unit tangent vector of each point.
The choice of which smooth curve is used as an approximation will not effect the rotation number.

Now consider the case where $S$ is a closed orientable surface with boundary $S$ of genus $g>0$.
Similarly to the genus zero case, we want to construct a closed curve in the plane for each path $P:i \pathto k$ where $(i,k) \in I \times K$.
We choose generators of the first homology group for the underlying closed surfaced without boundary.
The choice of homology generators does not affect our results as we will see in Section~\ref{sec:ind}.
The homology generators are chosen so that they do not intersect any vertices of $N$ and so that all intersections with edges of $N$ are transversal.
Also, the homology generators are chosen so that they intersect transversally with the cuts used to form $T$.
We then consider the \emph{punctured fundamental polygon} of $S$ which is the usual fundamental polygon of the underlying closed surface without boundary where the sides of the polygon correspond to homology generators, but we must remove some number of disks to create the boundary of the surface.
The punctured fundamental polygon has $4g$ sides with each corresponding to a homology generator, and when sides corresponding to the same homology generator are identified we obtain the surface $S$.
Note no vertex appears on any side of the punctured fundamental polygon and no edge or cut ever runs parallel to any side of the punctured fundamental polygon.
The punctured fundamental polygon represents a \emph{fundamental domain} of our surface.
See Figure~\ref{fig:PuncturedPoly} for an example of a punctured fundamental polygon.

\begin{figure}
\centering
\includegraphics{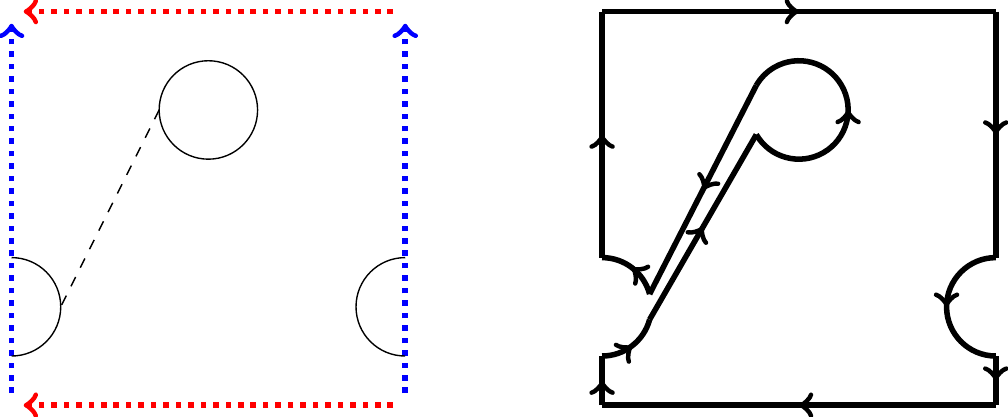}
\caption{On the left we have a punctured fundamental polygon for $S$ a torus with two boundary components with the cut between boundary components shown as a dashed line. On the right we have a closed curve around the boundary of $T$ drawn inside the fundamental domain.}
\label{fig:PuncturedPoly}
\end{figure}

When we have a network $N$ on $S$ with genus $g > 0$ we draw $N$ in the plane inside a single fundamental domain of $S$ and overlay the cuts used to construct the surface $T$.
Given any path $P:i \pathto j$ for $(i,j) \in I \times K$ we form a closed curve $C(P)$ in the plane, similarly to the genus zero case, by first traversing the path $P$ from $i$ to $j$ and then following the boundary of $T$ in our specified order from $j$ to $i$.
However, each time the closed curve leaves the chosen fundamental domain we connect the exit and entry points by following the sides of the punctured fundamental polygon clockwise from the exit point to the entry point.

Let $I \subseteq K$ be the collection of boundary vertices which are sources.
We consider the $I \times K$ matrix $A = A(N, \{ x_e\}_{e\in E})$ with entries given by $A_{ij} = M_{ij}$ for all $(i,j) \in I \times K$.
So, $A$ is obtained from $M$ by restricting to rows $I$ and columns $K$.
We also consider the $I \times K$ \emph{boundary measurement matrix} $B = B(N, \{x_e\}_{e \in E})$ with entries given by
$$B_{ij} = \sum_{P: i \pathto j} (-1)^{s_{ij} + r_P + 1} \wt(P)$$
for all $(i,j) \in I\times K$.
Here $s_{ij}$ denotes the number of elements of $I$ strictly between $i$ and $j$ with respect to $<$, and $r_P$ denotes the rotation number of $C(P)$.

This definition of the boundary measurement matrix for any closed orientable surface with boundary $S$ is due to Franco, Galloni, Penante, and Wen~\cite{FGPW15}.
Postnikov~\cite{Pos06} gave the original definition on the boundary measurement matrix in the case where the surface is a disk.
The boundary measurement matrix was considered for networks on the annulus by Gekhtman, Shapiro, and Vainshtein~\cite{GSV08} and for networks on any closed orientable genus zero surface with boundary by Franco, Galloni, and Mariotti~\cite{FGM14}.

Consider specializing the formal variables $x_e$ to real weights.
Notice the boundary measurement matrix is then a real $|I| \times |K|$ matrix of rank $|I|$.
Hence, for any directed network $N$ and choice of real weights, the boundary measurement matrix $B(N)$ describes an element of the real Grassmannian $\Gr(|I|, |K|)$.
This association of a directed network with real weights to an element of the Grassmannian is the known as the \emph{boundary measurement map}.
One feature of Postnikov's boundary measurement map applied to a directed network $N$ embedded in the disk is that when the edge weights are positive real numbers, the boundary measurement matrix $B(N)$ represents an element of the totally nonnegative Grassmannian.
The \emph{totally nonnegative Grassmannian} is defined to be elements of the Grassmannian such that all Pl{\"u}cker coordinates are nonnegative or nonpositive. 
That is, elements of the Grassmannian that can be represented by a matrix where each maximal minor is nonnegative.

When $S$ is a disk it is shown in~\cite{Pos06} that the maximal minors of $B(N)$ are subtraction-free rational expressions in the edge weights.
We call a directed network $N$ \emph{perfectly oriented} if each boundary vertex is a univalent source or sink and each interior vertex is trivalent and neither a source nor sink.
When a network is perfectly oriented, the interior vertices are of one of two types.
We distinguish the two types of interior vertices by coloring each interior vertex white or black.
White vertices have one incoming edge and two outgoing edges, and black vertices have two incoming edges and one outgoing edge.
For an example of a perfectly oriented network see Figure~\ref{fig:perfect}.

\begin{rem}
In~\cite{Pos06} it is shown how to transform any directed network $N$ on the disk to a perfectly oriented network $N'$ so that the boundary measurement matrix $B(N)$ is a specialization of the boundary measurement matrix of $B(N')$.
All transformations needed take place locally around a vertex, and hence will work on more general surfaces.
\label{rem:trans}
\end{rem}

If $N$ is perfectly oriented Talaska~\cite{Tal08} gives the following formula
\begin{equation}
\Delta_{I,J}(B) = \frac{\sum_{\bF \in \F_{I,J}(N)}\wt(\bF)}{\sum_{\bC \in \C(N)}\wt(\bC)}
\label{eq:Talaska2}
\end{equation}
where $J \subseteq K$ with $|I| = |J|$.
We notice Equation~(\ref{eq:Talaska2}) is very similar to Equation~(\ref{eq:Talaska1}) even though they describe minors of different matrices.
In Theorem~\ref{thm:sign} we will show that the boundary measurement matrix $B$ can be obtained from $A$ by a simple change of variables which explains the similarity of the formulas.
This theorem will also allow us to prove the following conjecture.

\begin{conj}[\cite{FGPW15}]
If $N = (V,E)$ is a perfectly oriented network embedded on a closed orientable surface with boundary, then for any $J \subseteq K$ with $|I| = |J|$
$$\Delta_{I,J}(B(N,\{x_e\}_{e \in E})) = \frac{\sum_{\bF \in \F_{I,J}(N)} \sigma(\bF) \wt(\bF)}{\sum_{\bC \in \C(N)} \sigma(\bC) \wt(\bC)}$$
for some $\sigma: \F_{I,J}(N) \cup \C(N) \to \{\pm 1 \}$.
\label{conj:flow}
\end{conj}

Our main result is that Conjecture~\ref{conj:flow} is true.
It follows from Equation~(\ref{eq:Talaska1}) and Theorem~\ref{thm:sign} which will be proven in the Section~\ref{sec:sign}.
Corollary~\ref{cor:flow} gives a formula for the maximal minors of the boundary measurement matrix where we explicitly describe the sign function $\sigma$ in Conjecture~\ref{conj:flow}.
Recall, if we specialize the formal variables to some real values, the boundary measurement matrix represents an element of the real Grassmannian.
In this context Conjecture~\ref{conj:flow} and Corollary~\ref{cor:flow} are formulas for the Pl{\"u}cker coordinates of this element of the Grassmannian.

\section{Boundary Measurement Independence}
\label{sec:ind}

Given a directed network $N$ embedded on a closed orientable surface with boundary $S$, we must make some choices when computing the boundary measurement matrix $B(N)$.
The first choice we must make is how to place the cuts on the surface $S$ to obtain the surface $T$ with a single boundary component.
The boundary measurement does depend on this choice.
For example, boundary measurement matrices are
\begin{align*}
B(N) &= \begin{bmatrix} 1 & x\end{bmatrix} & B(N') &= \begin{bmatrix}1 & -x\end{bmatrix}
\end{align*}
for the directed networks in Figure~\ref{fig:dependence}.

\begin{figure}
\centering
\scalebox{.7}{\includegraphics{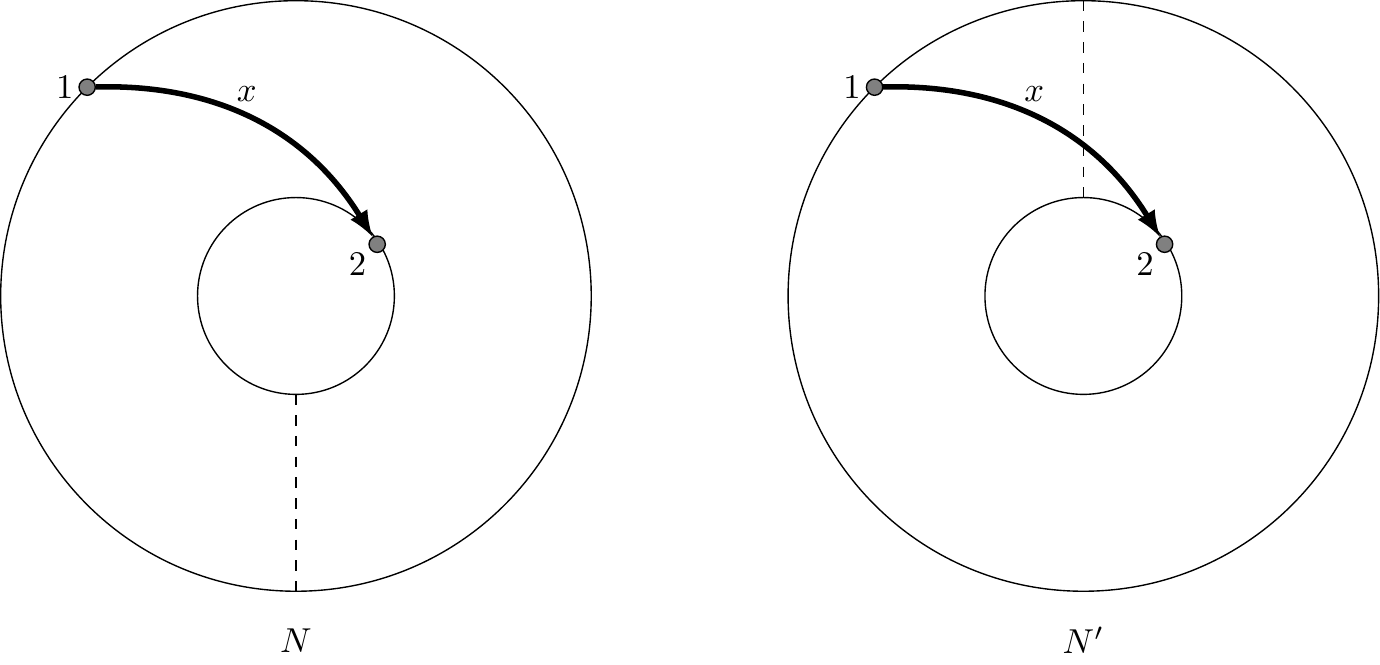}}
\caption{Two networks on the annulus each with a different choice of cut.}
\label{fig:dependence}
\end{figure}

Another choice we must make when computing the boundary measurement is how to represent the closed orientable surface with boundary in the plane.
For genus $g = 0$, we make a choice of which boundary component corresponds to the circle bounding the disk we draw our network inside.
For genus $g > 0$, we choose a fundamental domain.
In this section we will show that the boundary measurement does not depend on how we represent the surface in the plane.

Let $\psi: S^1 \to \mathbb{R}^2$ define a smooth closed curve $C$.
When $\psi(t_1) = \psi(t_2)$ for $t_1 \neq t_2$ we call $\psi(t_1)$ a \emph{self intersection point} of $C$.
If $\psi(t_1)$ is a self intersetion point of $C$ such that there exists a unique $t_2 \neq t_1$ with $\psi(t_1) = \psi(t_2)$ and $\{\psi'(t_1), \psi'(t_2)\}$ are linearly independent, we then call the self intersection point $\psi(t_1)$  \emph{simple}.
A smooth curve whose only self intersection points are simple is called \emph{normal}.
The rotation number of a normal curve differs in parity from the number of self intersections.
This was proven by Whitney in~\cite{Whi37} where it is also proven that any smooth curve can be transformed into a normal curve by small deformations without changing the curves rotation number.
When drawing closed curves inside a fundamental domain we sometimes may not connect exit and entry points along the sides of the punctured fundamental polygon, but rather draw some curve in the interior or exterior of the punctured fundamental polygon which has the same rotation number as the curve following the sides of the polygon.
This will be done to simplify the drawing of the curve and in some cases will be necessary to transform the curve into a normal curve.

Observe for any closed curve $C$ on a closed orientable surface with boundary $S$, we can construct a closed curve in the plane in the same way we do for the closed curves which come from paths in our oriented network.
Our next lemma will consider an arbitrary closed curve $C$ on $S$.
Given some representation of our surface in the plane, we will let $\hat{C}$ denote the corresponding closed curve in the plane.
Also, in the proof of the lemma we will consider a lift $C'$ of the closed curve $C$ to the universal cover of $S$ when the surface $S$ has genus $g > 0$.
Recall that each time a closed curve $C$ leaves the fundamental domain, we connected the exit and entry points along the boundary of the punctured fundamental polygon when constructing the closed curve $\hat{C}$ which lives in a single fundamental domain.
When doing this the tangent vector will make exactly one complete rotation.
To account for this we construct another curve $C''$ on the universal cover of $S$.
The curve $C''$ agrees with the curve $C'$ except we add a loop each time it crosses a homology generator.
See Figure~\ref{fig:closed} for an example of $C$, $\hat{C}$, $C'$ and $C''$.

\begin{figure}
\centering
\scalebox{.8}{\includegraphics{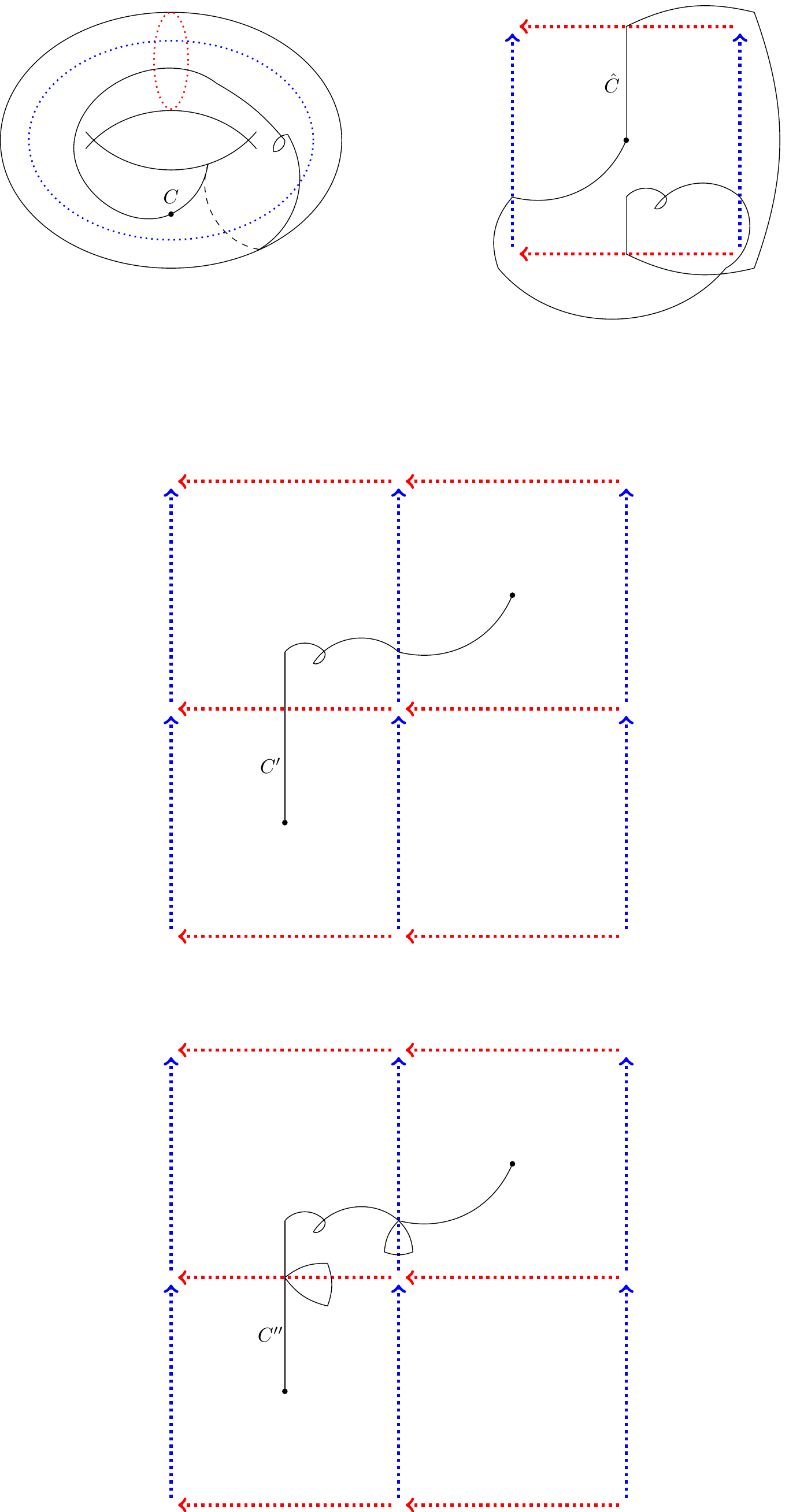}}
\caption{An example of a closed curve $C$ on the torus and the corresponding curves $\hat{C}$, $C'$ and $C''$.}
\label{fig:closed}
\end{figure}

\begin{lem}
Let $C$ be a closed curve on a closed orientable surface with boundary $S$  and let $\hat{C}$ be the closed curve corresponding to $C$ for some choice of representation of $S$ in the plane. 
The parity of the rotation number of $\hat{C}$ does not depend on the choice of representation of $S$ in the plane.
\label{lem:domain}
\end{lem}
\begin{proof}
Recall, the parity of the rotation number of a closed curve in the plane depends only on the number of self intersections of the curve.
For the case genus $g = 0$, it is clear the number of self intersections of $\hat{C}$ does not depend of the representation of $S$ in the plane.

Now consider the case genus $g > 0$.
We denote the rotation number of $\hat{C}$ by $r$.
We let $C'$ be a lift of $C$ to the universal cover of $S$.
The universal cover is homeomorphic to $\R^2$.
Let $h$ denote the number of times $C$ intersects any homology generator.
Notice the parity of $h$ is determined by the homology class of $C$, and hence the parity of $h$ is independent of the choice of fundamental domain used to represent $S$ in the plane.

If $C$ is null homologous, then $C'$ is a closed curve in $\R^2$.
If $C$ is not null homologous, then $C'$ is not a closed curve in $\R^2$.
However, we can still define the rotation number of $C'$ since the unit tangent vector at the starting point and ending point of $C'$ will be the same.
In any case, let $r'$ denote the rotation number of $C'$.
Notice each time $C$ intersects any homology generator, the tangent vector to the curve $\hat{C}$ we make a complete rotation in the clockwise direction  on the portion of the curve which  connects the entry and exit points of the fundamental domain.
We can modify the curve $C'$ by adding a small loop in the clockwise direction each time $C'$ intersects a lift of a homology generator.
We let $C''$ denote this modified curve.
We can construct $C''$ such that there is then a map $\phi: C'' \to \hat{C}$ such that $\vec{\mathbf{T}} = \vec{\mathbf{T}} \circ \phi$.
Recall, $\vec{\mathbf{T}}$ gives the unit tangent vector of a curve.
Let $r''$ denote the rotation number of $C''$
It then follows that $r = r''$ and that $r'' = r' + h$.
Therefore $r = r' + h$, and the parity of the rotation number of $\hat{C}$ is independent of how $S$ is represented in the plane.
\end{proof}

\begin{thm}
The boundary measurement of a directed network $N$ on a closed orientable surface with boundary $S$ is independent of how we represent $S$ in the plane.
\end{thm}
\begin{proof}
The only part of the boundary measurement matrix that depends on the representaion of $S$ in the plane is the rotation numbers $r_P$ of the closed curves $C(P)$ which correspond to paths $P$ in $N$.
In fact, the boundary measurement matrix only depends on the parity of $r_P$.
Therefore, the theorem then follows immediately from Lemma~\ref{lem:domain}.
\end{proof}

We have also made a choice to connect exit and entry points of a closed curve along the sides of the punctured fundamental polygon in the clockwise direction.
Observe the proof of Lemma~\ref{lem:domain} can easily be modified if we chose to connect the exit points in the counterclockwise direction.

\section{Signing Perfectly Oriented Networks}
\label{sec:sign}

In this section we prove a theorem which shows a relationship between the weighted path matrix and the boundary measurement matrix.
We show the boundary measurement matrix is the weighted path matrix with some edge weights thought of as negative.
That is, by replacing $x_e$ with $-x_e$ for some edges in $A(N, \{x_e\}_{e \in E})$ we obtain $B(N, \{x_e\}_{e \in E})$.
We first look at an example of signing the edges of a network.

\begin{figure}
\centering
\scalebox{.7}{\includegraphics{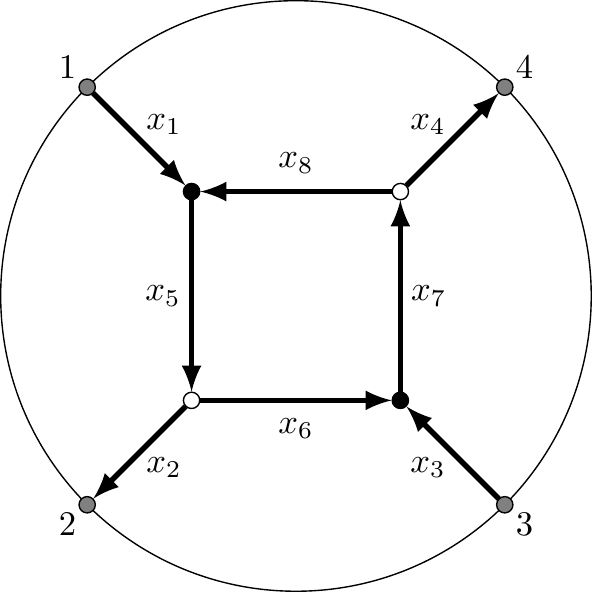}}
\caption{A perfectly oriented network on the disk.}
\label{fig:perfect}
\end{figure}

Let $N$ be the network in Figure~\ref{fig:perfect}. 
Here the boundary vertices are labeled to respect the usual ordering of the natural numbers so that $1<2<3<4$.
The weighted path matrix and boundary measurement matrix for $N$ are
$$
\begin{array}{c}
A=
\begin{bmatrix}
1 & \dfrac{x_1x_5x_2}{1 - x_5x_6x_7x_8} & 0 & \dfrac{x_1x_5x_6x_7x_4}{1 - x_5x_6x_7x_8} \\[15pt]
0 & \dfrac{x_3x_7x_8x_5x_2}{1 - x_5x_6x_7x_8} & 1 & \dfrac{x_3x_7x_4}{1 - x_5x_6x_7x_8}
\end{bmatrix}
\\[30pt]
B =
\begin{bmatrix}
1 & \dfrac{x_1x_5x_2}{1 + x_5x_6x_7x_8} & 0 & \dfrac{-x_1x_5x_6x_7x_4}{1+ x_5x_6x_7x_8} \\[15pt]
0 & \dfrac{x_3x_7x_8x_5x_2}{1 + x_5x_6x_7x_8} & 1 & \dfrac{x_3x_7x_4}{1 + x_5x_6x_7x_8}
\end{bmatrix}
\end{array}
$$
respectively.
Notice that $B$ can be obtained from $A$ by replacing $x_6$ with $-x_6$.
Theorem~\ref{thm:sign} shows that when $N$ is perfectly oriented $B$ can always be obtained from $A$ by a change a variable which gives each edge of $N$ a sign.
However, there is not a unique way to obtained $B$ from $A$.
For example, replacing $x_2$ and $x_5$ with $-x_2$ and $-x_5$ respectively is another possibility.
Theorem~\ref{thm:gauge} characterizes all possible ways to sign the edges of $N$.

Before stating the main theorem of this section we prove two lemmas which will be needed.

\begin{lem}
Let $N$ be a directed network embedded on a closed surface with boundary $S$ and let $T$ be the surface obtained after making cuts.
If $r_T$ is the rotation number of the closed curve which following the boundary of $T$ in a chosen fundamental domain, then $r_T \equiv 1 \pmod{2}$.
\label{lem:T}
\end{lem}
\begin{proof}
For genus $g = 0$ it is clear that $r_T \equiv 1 \pmod{2}$.
For genus $g > 0$ we can choose homology generators so that they do not intersect the boundary of $T$.
In this case it is again clear the $r_T \equiv 1 \pmod{2}$.
The general case for genus $g > 0$ then follows from Lemma~\ref{lem:domain}.
\end{proof}

For any path $P:i \pathto j$ we can form the closed curve $C'(P)$ by traversing the path $P$ from $i$ to $j$ and then following the boundary of $T$ from $j$ to $i$ opposite to our choosen direction. We let $r'_P$ denote the rotation number of $C'(P)$.
Our next lemma shows that we can use $r'_P$ in place of $r_P$ and the boundary measurement matrix will not change.

\begin{lem}
Let $N$ be a directed network embedded on a closed surface $S$ with boundary and $P$ is a path in $N$, then $r'_P \equiv r_P \pmod{2}$.
\label{lem:anti}
\end{lem}

\begin{proof}
Let $P:i \pathto j$ be a path from $i$ to $j$ in $N$.
We claim $r_P - r'_P = r_T \pm 1$.
To see this draw $C(P)$ and $C'(P)$ together in the same fundamental domain.
We then reverse the direction of $C'(P)$ and observe that we traverse the boundary of $T$  once and also traverse the path $P$ once from $i$ to $j$ as well as once in reverse from $j$ to $i$.
Hence we can compute $r_P - r'_P$ by considering the rotation number of the closed curve obtained by  first traversing $P$, then traversing the boundary of $T$, and finally traversing the path $P$ in reverse.
Thus $r_P - r'_P = r_T \pm 1$ and it follows by Lemma~\ref{lem:T} that $r'_P \equiv r_P \pmod{2}$.
\end{proof}

So, Lemma~\ref{lem:anti} shows that the direction in which we parameterize the boundary of $S$ does not affect the boundary measurement. We now state and prove our theorem on signing edges.

\begin{thm}
If $N = (V,E)$ is a perfectly oriented network embedded on a closed orientable surface with boundary, then there exists a collection $\{\epsilon_e\}_{e \in E} \in \{\pm 1\}^E$ such that
$$B(N,\{x_e\}_{e \in E}) = A(N, \{\epsilon_e x_e\}_{e \in E}).$$
\label{thm:sign}
\end{thm}

\begin{proof}
Let $N$ be a perfectly oriented network with vertex set $V$ and edge set $E$.
To show $B(N,\{x_e\}) = A(N, \{\epsilon_e x_e\})$ it suffices to show that the path $P:i \pathto j$ for any $(i,j) \in I \times K$ has the following property:
\begin{equation}
\wt(P)|_{\{\epsilon_e x_e\}_{e \in E}} = (-1)^{s_{ij} + r_P + 1} \wt(P)
\tag{$\diamond$}
\label{eq:diamond}
\end{equation}
When this is true for a choice of signs $\{\epsilon_e\}_{e\in E}$ we will say the path $P$ has property~(\ref{eq:diamond}).
We assume the network has at least one boundary source, or else there is no boundary measurement matrix.

Recall that $K$ denotes the set of boundary vertices of $N$, and we have an ordering of the boundary vertices.
We fix the following notation, if $j \in K$ is a boundary vertex we let $e_j$ denote the unique external edge which is incident on $j$ and write $\epsilon_j$ for $\epsilon_{e_j}$.
It can happen that $e_{j_1} = e_{j_2}$ for $j_1 \ne j_2$, in this case we will consider distinct signs $\epsilon_{j_1}$ and $\epsilon_{j_2}$ on half edges with the sign on the edge being the product $\epsilon_{j_1} \epsilon_{j_2}$.
We induct on the number of interior vertices.
If there are no interior vertices, then the result is true since each path consists of a single edge.

For the inductive step we chose any boundary source $i_0$ and construct a network $\tilde{N}$ with one fewer interior vertex.
The edge set of $\tilde{N}$ will be denoted $\tilde{E}$.
We will inductively chose signs $\{\tilde{\epsilon}_e\}_{e \in \tilde{E}}$ so that each path in $\tilde{N}$ has property~(\ref{eq:diamond}) and show how to modify these signs to give a collection $\{\epsilon_e\}_{e \in E}$ so that each path in the $N$  has property~(\ref{eq:diamond}).
Recall that for two boundary vertices $i$ and $j$ of $N$ we let $s_{ij}$ denote the number of boundary sources strictly between them in $N$.
For two  boundary vertices $i$ and $j$ of $\tilde{N}$ we let $\tilde{s}_{ij}$ denote the number of boundary sources strictly between them in $\tilde{N}$.
The inductive step falls into one of three cases depending on the boundary source $i_0$ and its unique neighboring vertex.

If $i_0$ is adjacent to a white vertex with outgoing edges $e'$ and $e''$ we then remove the white vertex and split $i_0$ into two boundary sources $i'_0 < i''_0$ as shown in Figure~\ref{fig:SplitWhite}.
Choose signs $\{\tilde{\epsilon}_e\}$ for the edges of $\tilde{N}$ by induction so that all paths in $\tilde{N}$ have property~(\ref{eq:diamond}).
We define the signs $\{\epsilon_e\}$ as follows
\begin{center}
\begin{tabular}{ll}
 $\epsilon_j = \tilde{\epsilon}_j$ &for $j \in K$ with $j < i_0$\\
$\epsilon_{e'} = \tilde{\epsilon}_{e'}$ & \\
$\epsilon_{i_0} = +1$ & \\
$\epsilon_{e''} = -\tilde{\epsilon}_{e''}$ \\
$\epsilon_j =  -\tilde{\epsilon}_j$ &for $j \in K$ with $j > i_0$\\
$\epsilon_e = \tilde{\epsilon}_e$ &otherwise
\end{tabular}
\end{center}
and now verify the collection of signs $\{\epsilon_e\}$  are valid.

\begin{figure}
\centering
\scalebox{.7}{\includegraphics{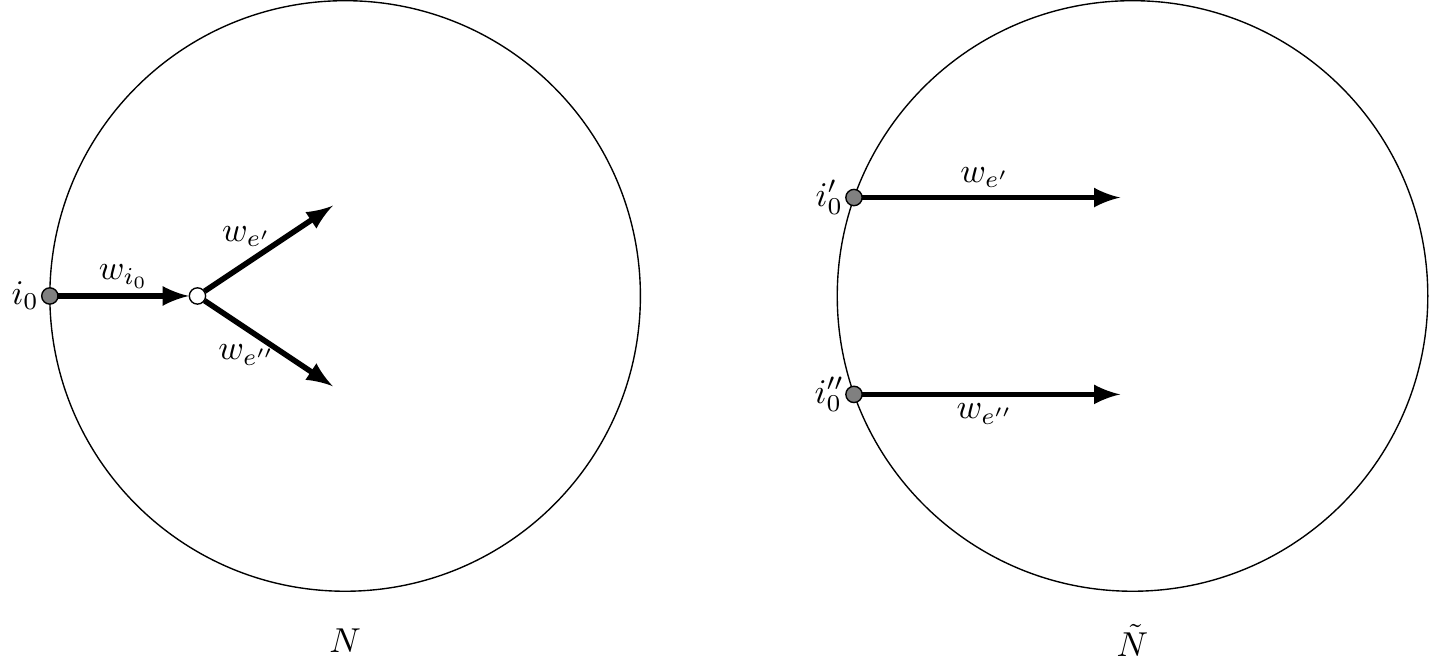}}
\caption{Splitting a white vertex}
\label{fig:SplitWhite}
\end{figure}

Consider a path $P: i \pathto j$ in $N$ with $i \ne i_0$.
The path $P$ corresponds to a path $\tilde{P}:i \pathto j$ in $\tilde{N}$ with $r_P = r_{\tilde{P}}$.
If $i,j < i_0$, then $s_{ij} = \tilde{s}_{ij}$ and $P$ has property~(\ref{eq:diamond}) since the modification does not introduce any sign change to $P$.
If $i < i_0 < j$ or $j < i_0 < i$, then $s_{ij} = \tilde{s}_{ij} - 1$ and $P$ has property~(\ref{eq:diamond}) since the modification introduces one sign change to $P$.
If $i,j > i_0$, then $s_{ij} = \tilde{s}_{ij}$ and $P$ has property~(\ref{eq:diamond}) since the modification introduces two sign changes to $P$.

Next consider a path $P:i_0 \pathto j$ in $N$. 
The path $P$ corresponds either to a path $\tilde{P}':i'_0 \pathto j$ or $\tilde{P}'':i''_0 \pathto j$.
First consider the case $P$ corresponds to $\tilde{P}'$.
If $j < i_0$, then $s_{ij} = \tilde{s}_{ij}$ and $P$ has property~(\ref{eq:diamond}) since the modification does not introduce any sign change to $P$.
If $i_0 < j$, then $s_{ij} = \tilde{s}_{ij} - 1$  and $P$ has property~(\ref{eq:diamond}) since the modification introduces one sign change to $P$.
Next consider the case $P$ corresponds to $\tilde{P}''$.
If $j < i_0$, then $s_{ij} = \tilde{s}_{ij}-1$ and $P$ has property~(\ref{eq:diamond}) since the modification introduces one sign change to $P$.
If $i_0 < j$, then $s_{ij} = \tilde{s}_{ij}$ and $P$ has property~(\ref{eq:diamond})  since the modification introduces two sign changes to $P$.
Therefore the signs $\{\epsilon_e\}$ are valid in this case.

If $i_0$ is adjacent to a black vertex we then  remove the black vertex and split $i_0$ into two boundary vertices $i'_0 < i''_0$ one of which will be a sink and the other of which will be a source.
We now consider the case where $i'_0$ is a sink and $i''_0$ is a source as shown in Figure~\ref{fig:SplitBlack1}.
Choose signs $\{\tilde{\epsilon}_e\}$ for the edges of $\tilde{N}$ by induction so that all paths in $\tilde{N}$ have property~(\ref{eq:diamond}).
We define the signs $\{\epsilon_e\}$ as follows
\begin{center}
\begin{tabular}{ll}
$\epsilon_{e'} = -\tilde{\epsilon}_{e'}$ & \\
$\epsilon_{i_0} = +1$ & \\
$\epsilon_e = \tilde{\epsilon}_e,$ &otherwise
\end{tabular}
\end{center}
and now verify the collection of signs $\{\epsilon_e\}$  as defined satisfy our rule. 

\begin{figure}
\centering
\scalebox{.7}{\includegraphics{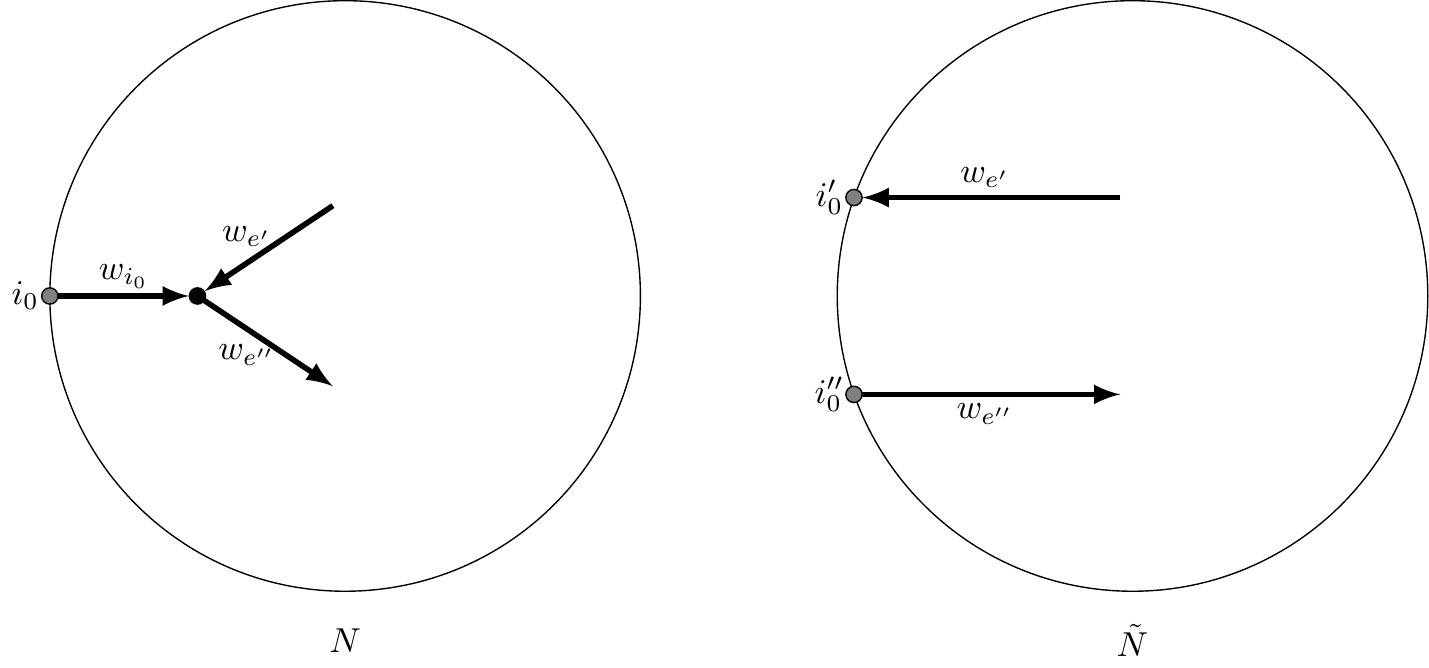}}
\caption{Splitting a black vertex in the case that $i'_0$ is a sink and $i''_0$ is a source}
\label{fig:SplitBlack1}
\end{figure}

Consider a path $P: i \pathto j$ in $N$ with $i \ne i_0$.
If $P$ does not use the edge $e'$, then $P$ corresponds to a path $\tilde{P}:i \to j$ in $\tilde{N}$ with $r_{P} = r_{\tilde{P}}$ and $s_{ij} = \tilde{s}_{ij}$.
If this is the case, then it is clear $P$ has property~(\ref{eq:diamond}).
Otherwise $P$ traverses the edge $e'$ some number of times.
Let $l+1$ be the number of times $P$ traverses $e'$ for $l \geq 0$.
The path $P$ corresponds to the concatenation of paths $\tilde{P}':i \pathto i'_0$, $\tilde{P}_k:i''_0 \pathto i'_0$ for $1 \leq k \leq l$, and $\tilde{P}'':i''_0 \pathto j$.
Now the sign of the product of the weights of these paths in $\tilde{N}$ is
$$(-1)^{\tilde{s}_{i i'_0} + r_{\tilde{P}'} + 1} (-1)^{\sum_{k=1}^{l} (r_{\tilde{P}_k} + 1)}(-1)^{\tilde{s}_{i''_0 j} +  r_{\tilde{P}''} + 1}$$
since $\tilde{s}_{i'_0 i''_0} = 0$. 
The sign of the path $P$ in $N$ will be
$$(-1)^{\tilde{s}_{i i'_0} + r_{\tilde{P}'} + 1} (-1)^{\sum_{k=1}^{l} (r_{\tilde{P}_k} + 1)}(-1)^{\tilde{s}_{i''_0 j} +  r_{\tilde{P}''} + 1}(-1)^{l+1}$$
since we pick up an addition factor of $-1$ each time we traverse $e'$.
Simplifying the sign of $P$ is
$$(-1)^{\tilde{s}_{i i'_0} + \tilde{s}_{i''_0 j} + 1+ r_{\tilde{P}'} + \sum_{k=1}^{l} r_{\tilde{P}_k}  + r_{\tilde{P}''}}. $$
We observe that
\begin{center}
\begin{tabular}{ll}
$s_{ij} =\tilde{s}_{i i'_0} + \tilde{s}_{i''_0 j} + 1$  &if $i \prec i_0 \prec j$\\
$s_{ij} = \tilde{s}_{i i'_0} + \tilde{s}_{i''_0 j}$ & if $j \prec i_0 \prec i,$\\
\end{tabular}
\end{center}
and so the sign of $P$ is
\begin{center}
\begin{tabular}{ll}
$(-1)^{s_{ij} + r_{\tilde{P}'} + \sum_{k=1}^{l} r_{\tilde{P}_k}  + r_{\tilde{P}''}}$ & if $i \prec i_0 \prec j$\\
$(-1)^{s_{ij} + 1+ r_{\tilde{P}'} + \sum_{k=1}^{l} r_{\tilde{P}_k}  + r_{\tilde{P}''}}$ &if  $j \prec i_0 \prec i.$
\end{tabular}
\end{center}
Finally we observe that
\begin{center}
\begin{tabular}{ll}
$r_P + 1 \equiv r_{\tilde{P}'} + \sum_{k=1}^{l} r_{\tilde{P}_k}  + r_{\tilde{P}''} \pmod{2}$ & if $i \prec i_0 \prec j$\\
$r_P \equiv r_{\tilde{P}'} + \sum_{k=1}^{l} r_{\tilde{P}_k}  + r_{\tilde{P}''} \pmod{2}$ & if $j \prec i_0 \prec i$\\
\end{tabular}
\end{center}
and it follows that $P$ has property~(\ref{eq:diamond}).
See Figure~\ref{fig:close} for the case of the disk.
More generally when the surface is not the disk the boundary will still be a circle and Lemma~\ref{lem:T} shows that the rotation number of traversing the boundary will always be odd, and hence can be thought of as shown in Figure~\ref{fig:close}.

\begin{figure}
\centering
\scalebox{.7}{\includegraphics{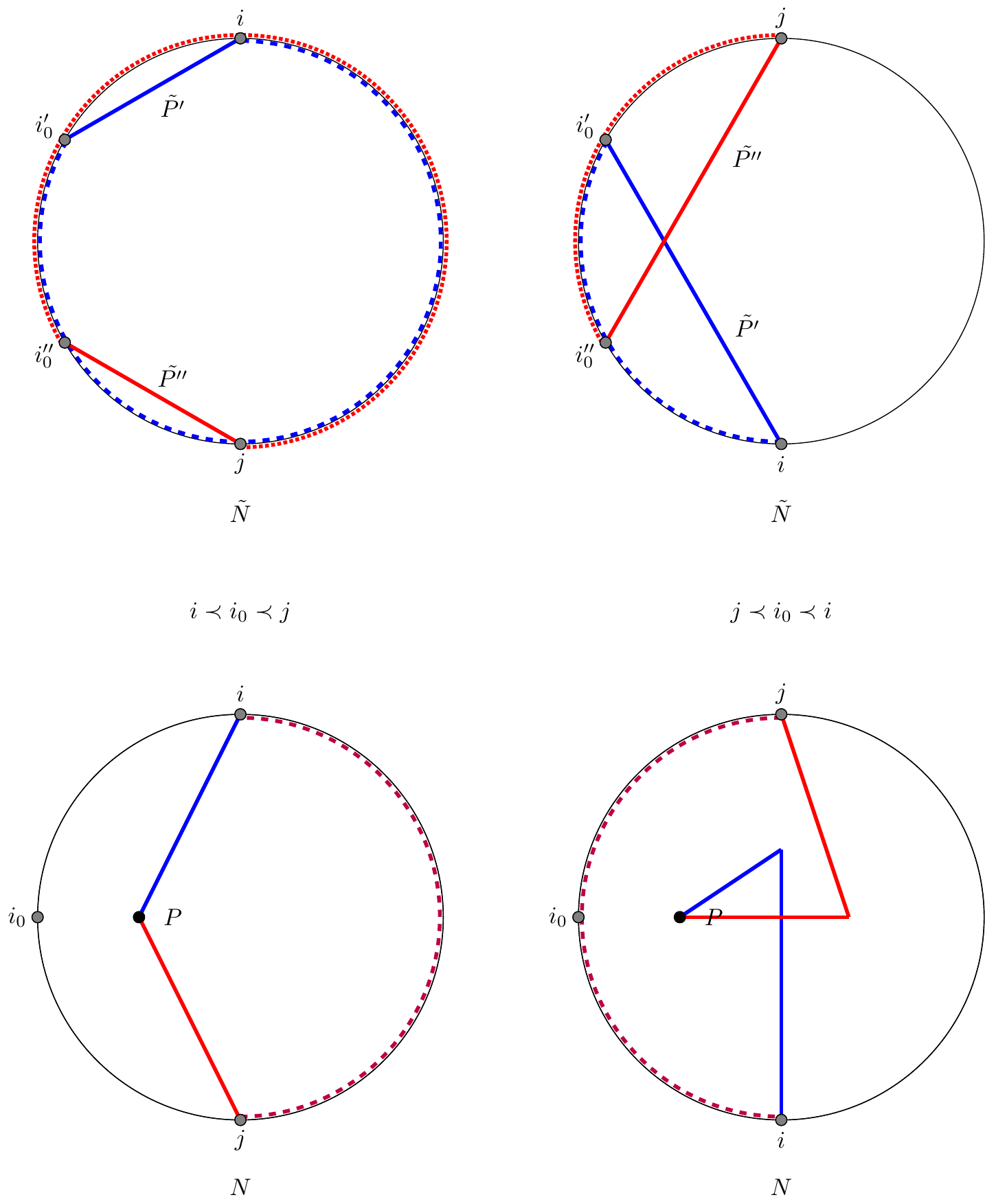}}
\caption{Closing paths in $\tilde{N}$. The case $i \prec i_0 \prec j$ is shown on the right while the case $j \prec i_0 \prec i$ is shown on the left.}
\label{fig:close}
\end{figure}

Now consider a path $P: i_0 \pathto j$ in $N$.
If $P$ does not use the edge $e'$, then $P$ corresponds to a path $\tilde{P}:i''_0 \to j$ in $\tilde{N}$ with $r_{P} = r_{\tilde{P}}$ and $s_{ij} = \tilde{s}_{i''_0 j}$.
If this is the case, then it is clear $P$ has property~(\ref{eq:diamond}).
Otherwise $P$ traverses the edge $e'$ some number of times.
Let $l$ be the number of times $P$ traverses $e'$ for $l > 0$.
In this case $P$ corresponds to the concatenation of paths $\tilde{P}_k:i''_0 \pathto i'_0$ for $1 \leq k \leq l$, and $\tilde{P}'':i''_0 \pathto j$.
Now the sign of the product of the weights of these paths in $\tilde{N}$ is
$$(-1)^{\sum_{k=1}^{l} (r_{\tilde{P}_k} + 1)} (-1)^{\tilde{s}_{i''_0 j} + r_{\tilde{P}''} + 1}$$
since $\tilde{s}_{i'_0 i''_0} = 0$. 
The sign of the path $P$ in $N$ will be
$$(-1)^{\sum_{k=1}^{l} (r_{\tilde{P}_k} + 1)} (-1)^{\tilde{s}_{i''_0 j} +  r_{\tilde{P}''} + 1} (-1)^l$$
since we pick up an addition factor of $-1$ each time we traverse $e'$.
Simplifying, the sign of $P$ is
$$(-1)^{s_{ij} + \sum_{k=1}^{l} r_{\tilde{P}_k}  + r_{\tilde{P}''} + 1} $$
since $s_{ij} = \tilde{s}_{i''_0 j}$.
The equality
$$r_P = \sum_{k=1}^{l} r_{\tilde{P}_k}  + r_{\tilde{P}''}$$
implies that $P$ has property~(\ref{eq:diamond}).

The final case is again $i_0$ is adjacent to a black vertex, and we remove the black vertex and split $i_0$ into two boundary vertices $i'_0 < i''_0$.
This time we consider the case where $i'_0$ is a source and $i''_0$ is a sink as shown in Figure~\ref{fig:SplitBlack2}.
This case will be identical to the previous case of splitting a black vertex, after applying Lemma~\ref{lem:anti} and forming closed curves in the opposite direction,  with the subcases  $i \prec i_0 \prec j$ and $j \prec i_0 \prec i$ reversed.
\begin{figure}
\centering
\scalebox{.7}{\includegraphics{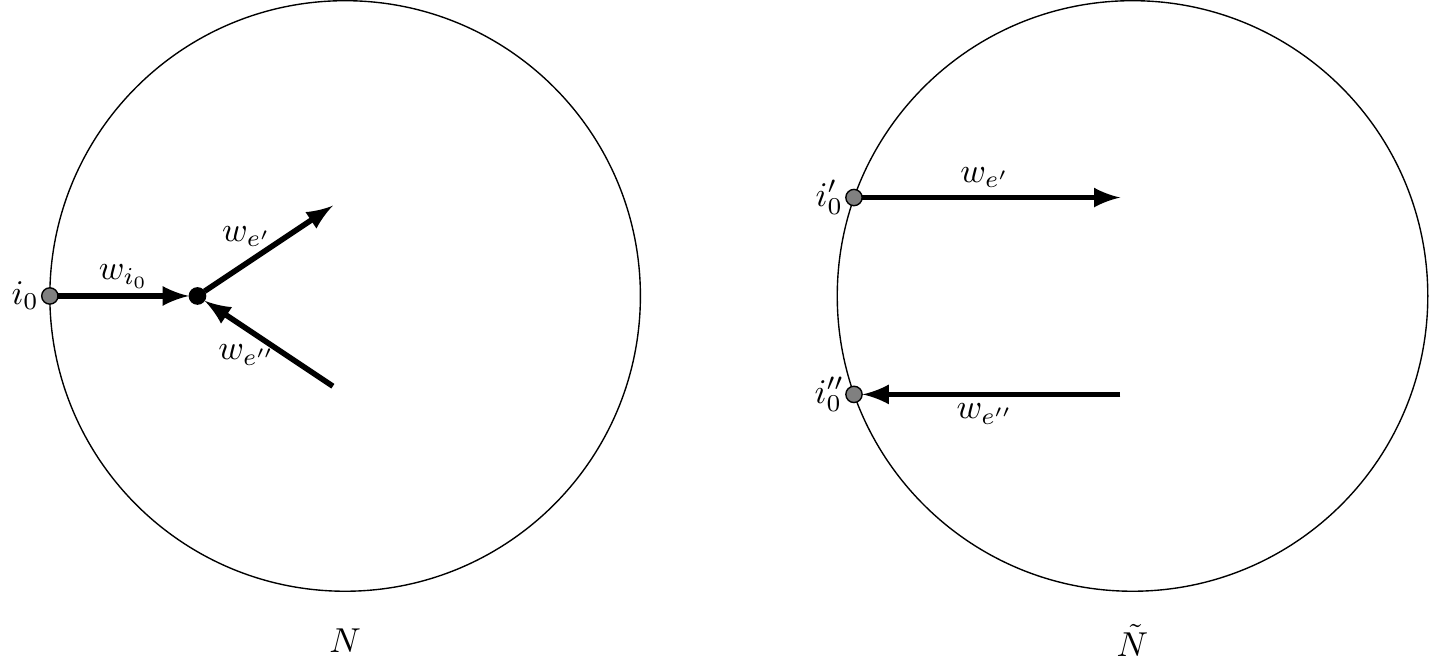}}
\caption{Splitting a black vertex in the case that $i'_0$ is a source and $i''_0$ is a sink}
\label{fig:SplitBlack2}
\end{figure}
\end{proof}

\begin{figure}
\centering
\scalebox{.7}{\includegraphics{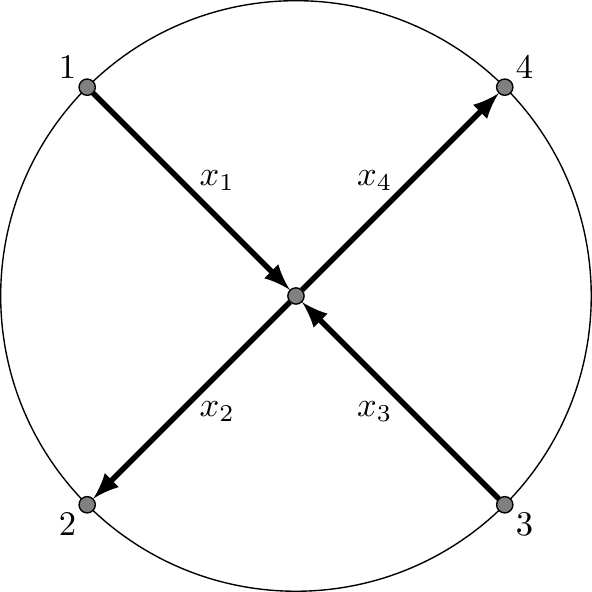}}
\caption{A network which is not perfectly oriented}
\label{fig:non}
\end{figure}

Theorem~\ref{thm:sign} need not be true when $N$ is not a perfectly oriented network.
See Figure~\ref{fig:non} for an example of a network for which Theorem~\ref{thm:sign} does not hold.
The boundary measurement matrix for the network in Figure~\ref{fig:non}  is
$$B = 
\begin{bmatrix}
1 & x_1 x_2 & 0 & -x_1 x_4\\
0 & x_2 x_3 & 1 & x_3 x_4
\end{bmatrix}.
$$
The matrix $B$ cannot be obtained from the weighted path matrix for this example.
Notice the second column of $B$ would require $x_1$ and $x_3$ receive the same sign, while the fourth column of $B$ would require $x_1$ and $x_3$ receive opposite signs.
However, as mentioned in Remark~\ref{rem:trans} the network in Figure~\ref{fig:non} can be transformed to a perfectly oriented network.
In this case it turns out the perfectly oriented network we get after the transformation is the network in Figure~\ref{fig:perfect} for which we have already seen how to sign the edges.
We include Algorithm~\ref{alg:sign} for finding a signing of edges as in Theorem~\ref{thm:sign}.
This recursive algorithm exactly corresponds to the induction used in the proof.

\begin{algorithm}
\caption{Signing Edges of a Perfectly Oriented Network}\label{alg:sign}
\begin{algorithmic}
\Require A perfectly oriented network $N = (V, E)$.
\Function{FindSigns}{$N$}
\If{$\int(V) = \varnothing$}
\For{$e = (i,j) \in E$}
\State $\epsilon_e = (-1)^{s_{ij} + r_e + 1}$
\EndFor
\State \Return $\{\epsilon_e\}_{e \in E}$
\Else
\State Choose boundary source $i_0 \in I$ adjacent to some interior vertex.
\State Let $e_0 = (i_0, v_0)$ be the unique edge incident on $i_0$.
\State Let $e'$ and $e''$ be to two edges different from $e_0$ incident on $v_0$ with $e' < e''$.
\State $\tilde{N} \gets \Call{Split}{N,i_0,v_0,e_0,e',e''}$
\State $\{\tilde{\epsilon}_e\}_{e \in E(\tilde{N}) }\gets \Call{FindSigns}{\tilde{N}}$
\State \Return $\Call{ModifySigns}{N,i_0,v_0,e_0,e',e'', \{\tilde{\epsilon}_e\}_{e \in E(\tilde{N}) }}$
\EndIf
\EndFunction

\Function{Split}{$N$,$i_0$,$v_0$,$e_0$,$e'$,$e''$}
\State $\tilde{V} \gets (V \setminus \{i_0, v_0\}) \cup \{i'_0, i''_0\}$
\State $\tilde{K} = (K \setminus \{i_0\}) \cup \{i'_0, i''_0\}$
\State Order $\tilde{K}$ by $i_1 < i'_0 < i''_0 < i_2$ for all $i_1, i_2 \in K$ such that $i_1 < i_0 < i_2$
\If{$e' = (v_0,x)$} $\tilde{e}' \gets (i'_0,x)$ \EndIf
\If{$e' = (x,v_0)$} $\tilde{e}' \gets (x,i'_0)$ \EndIf
\If{$e'' = (v_0,x)$} $\tilde{e}'' \gets (i''_0,x)$ \EndIf
\If{$e'' = (x,v_0)$} $\tilde{e}'' \gets (x,i''_0)$ \EndIf
\State $\tilde{E} \gets (E \setminus \{e_0,e',e''\}) \cup \{\tilde{e}', \tilde{e}''\}$
\State $\tilde{N} \gets (\tilde{V}, \tilde{E})$
\State \Return $\tilde{N}$
\EndFunction

\Function{ModifySigns}{$N,i_0,v_0,e_0,e',e'', \{\tilde{\epsilon}_e\}_{e \in E(\tilde{N})}$}
\If{$\tilde{e}' = (v_0,x)$ and $\tilde{e}'' = (v_0, y)$}
\State $\epsilon_{i_0} \gets +1$
\State $\epsilon_{e'} \gets \tilde{\epsilon}_{\tilde{e}'}$
\State $\epsilon_{e''} \gets \tilde{\epsilon}_{\tilde{e}''}$
\For{$j \in K$ with $j < i_0$}
\State $\epsilon_{j} \gets \tilde{\epsilon}_j$
\EndFor
\For{$j \in K$ with $j > i_0$}
\State $\epsilon_{j} \gets -\tilde{\epsilon}_j$
\EndFor
\For{All other edges $e \in E$}
\State $\epsilon_e \gets \tilde{\epsilon}_e$
\EndFor
\EndIf
\If{$\tilde{e}' = (x,v_0)$ and $\tilde{e}'' = (v_0, y)$}
\State $\epsilon_{i_0} \gets +1$
\State $\epsilon_{e'} \gets -\tilde{\epsilon}_{\tilde{e}'}$
\State $\epsilon_{e''} \gets \tilde{\epsilon}_{\tilde{e}''}$
\For{All other edges $e \in E$}
\State $\epsilon_e \gets \tilde{\epsilon}_e$
\EndFor
\EndIf
\If{$\tilde{e}' = (v_0,x)$ and $\tilde{e}'' = (y,v_0)$}
\State $\epsilon_{i_0} \gets +1$
\State $\epsilon_{e'} \gets \tilde{\epsilon}_{\tilde{e}'}$
\State $\epsilon_{e''} \gets -\tilde{\epsilon}_{\tilde{e}''}$
\For{All other edges $e \in E$}
\State $\epsilon_e \gets \tilde{\epsilon}_e$
\EndFor
\EndIf
\EndFunction
\end{algorithmic}
\end{algorithm}

\section{A Formula for Pl{\"u}cker Coordinates}
\label{sec:Plucker}

Notice that from Theorem~\ref{thm:sign} it follows that Conjecture~\ref{conj:flow} is true.
We now want to give an explicit formula for the minors of the boundary measurement matrix.
In order to do this we must first review some concepts and results that can be found in~\cite{Pos06} and~\cite{Tal08}.
Take $I, J \subseteq [n]$ with $|I| = |J|$.
Let $\pi: I \to J$ be a bijection with $\pi(i) = i$ for all $i \in I \cap J$.
A pair $(i_1, i_2) \in I \times I$ where $i_1 < i_2$ is called a \emph{crossing} of $\pi$ if the following condition holds
$$(i_1 - \pi(i_2))(\pi(i_2) - \pi(i_1))(\pi(i_1) - i_2))(i_2 - i_1) < 0.$$
This condition is equivalent to the chord from $i_1$ to $\pi(i_1)$ crossing the chord from $i_2$ to $\pi(i_2)$ when the elements of $[n]$ are placed in cyclic order on the boundary of a disk.
Thinking of $I$ as a collection of boundary sources  and $J$ as a collection of boundary vertices, the condition of being a crossing means that a path $P_1: i \pathto \pi(i_1)$ must intersect any path $P_2:i_2 \pathto \pi(i_2)$ in any network embedded on a disk.
However, when our surface is not a disk it can happen that $(i_1, i_2)$ is a crossing of $\pi$ but paths  $P_1: i \pathto \pi(i_1)$ and $P_2:i_2 \pathto \pi(i_2)$ do not intersect.
See Figure~\ref{fig:crossing} for a pictorial representation of a crossing on the disk and an example of paths of the annulus which come from a crossing but do not intersect.
We let $\xing(\pi)$ denote the number of crossings of $\pi$.

\begin{figure}
\centering
\scalebox{.7}{\includegraphics{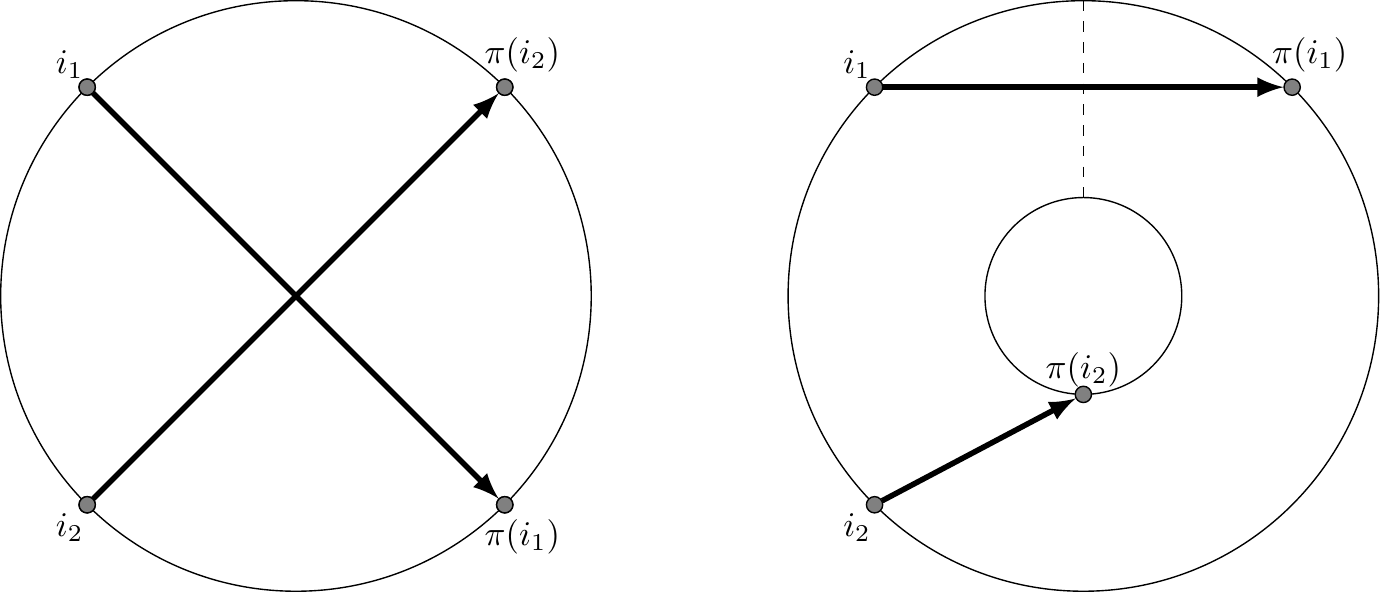}}
\caption{Crossings on the disk and the annulus. Here boundary vertices are ordered $i_1 < i_2 < \pi(i_1) < \pi(i_2)$.}
\label{fig:crossing}
\end{figure}

If $|I| = |J| = k$ then a bijection $\pi: I \to J$ determines a unique permutation $\pi \in S_k$ by standardizing $I$ and $J$.
We let $\inv(\pi)$ denote the number of inversion of $\pi$ when view as an element of $S_k$.
We let $s_{i,j}$ denote the number of elements of $I$ strictly between $i$ and $j$.

\begin{lem}[\cite{Tal08}]
If $I, J \subset [n]$ with $|I| = |J|$ and $\pi: I \to J$ is a bijection such that $\pi(i) = i$ for all $i \in I \cap J$, then
$$(-1)^{\xing(\pi)} = (-1)^{\inv(\pi)} \prod_{i \in I} (-1)^{s_{i, \pi(i)}}.$$
\label{lem:xing}
\end{lem}

\begin{proof}
This is shown during the proof of~\cite[Proposition 2.12]{Tal08}.
\end{proof}

We now give our formula for the Pl{\"u}cker coordinates of the boundary measurement map.

\begin{cor}
If $N = (V,E)$ is a perfectly oriented network embedded on a closed orientable surface with boundary, then for any $J \subseteq K$ with $|I| = |J|$
$$\Delta_{I,J}(B(N,\{x_e\}_{e \in E})) = \frac{\sum_{\bF \in \F_{I,J}(N)} (-1)^{c(\bF)}\wt(\bF)}{\sum_{\bC \in \C(N)} \wt(\bC)}$$
where $c(\bF) = \xing(\pi) + \sum_{P \in \bP} (r_P + 1)$ if $\bF = (\bP,\bC)$.
\label{cor:flow}
\end{cor}
\begin{proof}
We first take $\{\epsilon_e\}_{e \in E}$ such that $B(N,\{x_e\}_{e \in E}) = A(N, \{\epsilon_e x_e\}_{e \in E})$ which necessarily exists by Theorem~\ref{thm:sign}.
Using Equation~(\ref{eq:Talaska1}) we obtain
$$\Delta_{I,J}(B(N,\{x_e\}_{e \in E})) = \frac{\sum_{\bF \in \F_{I,J}(N)} \sgn(\bF)\left(\prod_{e \in \bF} \epsilon_e \right)\wt(\bF)}{\sum_{\bC \in \C(N)} \sgn(\bC)\left(\prod_{e \in \bC} \epsilon_e \right)\wt(\bC)}.$$
Since traversing a cycle in a perfectly oriented network will always change the rotation number by exactly one, it follows that $\prod_{e \in \bC} \epsilon_e = (-1)^{|\bC|}$ for any $\bC \in \C(N)$.
Also, $\sgn(\bC) = (-1)^{|\bC|}$ for any $\bC \in \C(N)$ and thus 
$$\sgn(\bC)\left(\prod_{e \in \bC} \epsilon_e\right) = 1$$
and the denominator in the corollary is correct.

It remains to show the numerator in the corollary is correct.
That is we must show $\sgn(\bF)\left(\prod_{e \in \bF} \epsilon_e \right) = (-1)^{\xing(\pi) + \sum_{P \in \bP} (r_P + 1)}$ for any $\bF \in \F_{I,J}(N)$.
Take $\bF = (\bP, \bC) \in \F_{I,J}(N)$ and let $\pi: I \to J$ be the bijection determined by $\bP$, then
\begin{align*}
\sgn(\bF)\left(\prod_{e \in \bF} \epsilon_e \right) &= \sgn(\bP) \sgn(\bC) \left(\prod_{e \in \bP} \epsilon_e\right) \left(\prod_{e \in \bC} \epsilon_e\right)\\
&= \sgn(\pi) \left(\prod_{e \in \bP} \epsilon_e\right)\\
&= (-1)^{\inv(\pi)} \left(\prod_{i \in I} (-1)^{s_{i, \pi(i)}}\right) \left(\prod_{P \in \bP} (-1)^{r_P + 1}\right)\\
&= (-1)^{\xing(\pi) + \sum_{P \in \bP} (r_P + 1)}
\end{align*}
where we have made use of Lemma~\ref{lem:xing}.
\end{proof}

In the case our surface $S$ is a disk it is easy to see that the formula in Corollary~\ref{cor:flow} contains no negative terms.
On the disk for any flow $\bF = (\bP,\bC)$ we must have $\xing(\pi) = 0$ and $r_P = \pm 1$ for all $P \in \bP$.
Thus, $c(\bF)$ is even for any flow $\bF$ in the disk.
Hence we recover Equation~(\ref{eq:Talaska2}).
For more general surfaces we no longer have positivity, for example see Figure~\ref{fig:dependence}.

\section{The Gauge Action and Uniqueness of Signs}
\label{sec:gauge}

Given a directed network $N = (V,E)$ embedded on a surface the \emph{gauge group}  $\G = \G(N) := (\R^{*})^{\int(V)}$ where $\int(V) = V \setminus K$ denotes the set of interior vertices of $N$ and $\R^{*}$ denotes the nonzero real numbers.
We also define the \emph{weight space} $\X = \X(N)$ to be the set of all collections $\{a_e x_e\}_{e \in E}$ where $a_e \in \R^{*}$.
Notice here to each edge $e \in E$ we associate a nonzero real number $a_e$ and a formal variable $x_e$.
An element of the gauge group  $g=(g_v)_{v \in \int(V)} \in \G$ acts on an element of the weight space $X = \{ a_e x_e\}_{e \in E} \in \X$ as follows
$$g \cdot X  = \{(g \cdot a_e) x_e\}_{e \in E}$$
where if $e = (i,j)$ then $g \cdot a_e = g^{-1}_j a_e g_i$ (with the convention that $g_i = 1$ if $i \in K$ is a boundary vertex).
It follows that
$$A(N,X) = A(N, g \cdot X)$$
for all $g \in \G$ and $X \in \X$.
When $X,Y \in \X(N)$ are such that $Y = g \cdot X$ for some $g \in \G(N)$ we call $X$ and $Y$ \emph{gauge equivalent}.

\begin{algorithm}
\caption{Finding Gauge Transformation}\label{alg:gauge}
\begin{algorithmic}
\Require A directed network $N = (V,E)$ embedded on a closed orientable surface with boundary such that each vertex is contained in some path between boundary vertices and $X, Y\in \X(N)$ are such that $A(N,X) = A(N,Y)$.
\Function{FindGauge}{$X$,$Y$}
\State Let $X = \{a_ex_e\}_{e\in E}$ and $Y = \{b_ex_e\}_{e\in E}$
\State $g \gets (1)_{v \in \int(V)}$
\State $\mathcal{O} \gets \int(V)$
\State $\mathcal{C} \gets V \setminus \int(V)$
\While{$\mathcal{O} \ne \varnothing$}
\State Choose $(u,v) \in E$ such that $(u,v) \in \mathcal{C} \times \mathcal{O}$ 
\State $g_v \gets \frac{g \cdot a_{(u,v)}}{b_{(u,v)}}$
\State $\mathcal{C} \gets \mathcal{C} \cup \{v\}$
\State $\mathcal{O} \gets \mathcal{O} \setminus \{v\}$
\EndWhile
\State \Return $g$ \Comment{$g$ returned will be such that $g \cdot X = Y$}
\EndFunction
\end{algorithmic}
\end{algorithm}

\begin{thm}
Let $N = (V,E)$ be a directed network embedded on a closed orientable surface with boundary such that every vertex in contained in some path between boundary vertices, then $A(N,X) = A(N,Y)$ for $X,Y \in \X(N)$ if and only if $X$ and $Y$ are gauge equivalent.
\label{thm:gauge}
\end{thm}
\begin{proof}
Our proof will show that Algorithm~\ref{alg:gauge} returns $g \in \G(N)$ such that $g \cdot X = Y$.
Let $X = \{a_e x_e\}_{e \in E}$ and $Y = \{b_e x_e\}_{e \in E}$.
First note that Algorithm~\ref{alg:gauge} will always terminate since each vertex of $N$ is contained in some path between boundary vertices and $\mathcal{C}$ initially consists of all the boundary vertices.
Furthermore, when the algorithm terminates $\mathcal{C} = V$.
Also, observe that if $v \in \mathcal{C} \cap \int(V)$ at some stage of the algorithm there is a directed path from some boundary vertex to the vertex $v$ passing through only vertices in $\mathcal{C}$.
Lastly, we note that at a given stage of the algorithm $g_v = 1$ whenever $v \not\in \mathcal{C}$.
It suffices to show that at each step of Algorithm~\ref{alg:gauge} we have the following property:
\begin{equation}
g\cdot a_e = b_e \mathrm{\;for\;all\;} e \in \mathcal{C} \times \mathcal{C}
\tag{$\star$}
\label{eq:star}
\end{equation}

Initially $\mathcal{C}$ consists of only the boundary vertices and $g = (1)_{v \in \int(V)}$.
At this stage we have $g \cdot a_e = a_e$ for all $e \in E$, and $a_e = b_e$ whenever $e \in \mathcal{C} \times \mathcal{C}$ by the assumption that $A(N,X) = A(N,Y)$.
So, initially we have property~(\ref{eq:star}).

We now consider extending the set of vertices $\mathcal{C}$.
Suppose we are at some stage of the algorithm where $g \cdot a_e = b_e$ for all $e \in \mathcal{C} \times \mathcal{C}$.
Consider $(u,v) \in \mathcal{C} \times \mathcal{O}$ and let $\mathcal{C}' = \mathcal{C} \cup \{v\}$ and $g'$ be such that $g'_v = g \cdot a_{(u,v)}/b_{(u,v)}$ and $g'_x = g_x$ for $x \ne v$.
Now we must show for all $e \in \mathcal{C}' \times \mathcal{C}'$ that $g' \cdot a_e = b_e$.
We need only consider edges $e \in \mathcal{C}' \times \mathcal{C}'$  incident on $v$ as $g' \cdot a_e = g \cdot a_e = b_e$ for $e \in \mathcal{C}' \times \mathcal{C}'$ with $e$ not incident on $v$.
First we compute
\begin{align*}
g' \cdot a_{(u,v)} &= (g'_v)^{-1} a_{(u,v)} g'_u\\
&= \frac{b_{(u,v)} a_{(u,v)} g'_u}{g \cdot a_{(u,v)}}\\
&= \frac{b_{(u,v)} (g \cdot a_{(u,v)})}{g \cdot a_{(u,v)}}\\
&= b_{(u,v)}
\end{align*}
and conclude $g' \cdot a_{(u,v)} = b_{(u,v)}$.

Consider $(w,v) \in E$ such that $w \in \mathcal{C}$.
We can find paths $P_u:i_u \pathto u$ and $P_w:i_w \pathto w$ passing through only vertices of $\mathcal{C}$ for $i_u,i_w \in I$.
Choose some path $P:v \pathto j$ for $j \in K$ so we get paths $P_1 = P_u (u,v) P:i_u \pathto j$ and $P_2 = P_w (w,v) P: i_w \pathto j$.
It then follows that 
\begin{align*}
\prod_{e \in P_1} a_e &= \prod_{e \in P_1} b_e & \prod_{e \in P_2} a_e &= \prod_{e \in P_2} b_e\\
\end{align*}
and so also
\begin{align*}
\prod_{e \in P_1} g' \cdot a_e &= \prod_{e \in P_1} b_e &\prod_{e \in P_2} g' \cdot a_e &= \prod_{e \in P_2} b_e.
\end{align*}
Considering ratios we see
\begin{equation*}
\frac{\left(\prod_{e \in P_u} g' \cdot a_e\right) (g' \cdot a_{(u,v)}) \left(\prod_{e \in P} g'\cdot a_e\right)}{\left(\prod_{e \in P_w} g' \cdot a_e\right) (g' \cdot a_{(w,v)}) \left(\prod_{e \in P} g'\cdot a_e\right)} = \frac{\left(\prod_{e \in P_u} b_e\right) (b_{(u,v)}) \left(\prod_{e \in P} b_e\right)}{\left(\prod_{e \in P_w} b_e\right) (b_{(w,v)}) \left(\prod_{e \in P} b_e\right)}
\end{equation*}
and recalling $g' \cdot a_e = b_e$ for $e \in P_u \cup P_w \subseteq \mathcal{C}$ we can conclude $g' \cdot a_{(w,v)} = b_{(w,v)}$ as desired.

Consider $(v,w) \in E$ such that $w \in \mathcal{C}$.
We can find paths $P_u:i_u \pathto u$ and $P_w:i_w \pathto w$ passing through only vertices of $\mathcal{C}$ for $i_u,i_w \in I$.
Choose some path $P:w \pathto j$ for $j \in K$ so we get paths $P_1 = P_u (u,v)(v,w) P:i_u \pathto j$ and $P_2 = P_w P: i_w \pathto j$.
It then follows that 
\begin{align*}
\prod_{e \in P_1} a_e &= \prod_{e \in P_1} b_e & \prod_{e \in P_2} a_e &= \prod_{e \in P_2} b_e\\
\end{align*}
and so also
\begin{align*}
\prod_{e \in P_1} g' \cdot a_e &= \prod_{e \in P_1} b_e &\prod_{e \in P_2} g' \cdot a_e &= \prod_{e \in P_2} b_e.
\end{align*}
Considering ratios we see
\begin{equation*}
\frac{\left(\prod_{e \in P_u} g' \cdot a_e\right) (g' \cdot a_{(u,v)})(g' \cdot a_{(v,w)}) \left(\prod_{e \in P} g'\cdot a_e\right)}{\left(\prod_{e \in P_w} g' \cdot a_e\right)\left(\prod_{e \in P} g'\cdot a_e\right)} = \frac{\left(\prod_{e \in P_u} b_e\right) (b_{(u,v)})(b_{(v,w)}) \left(\prod_{e \in P} b_e\right)}{\left(\prod_{e \in P_w} b_e\right) \left(\prod_{e \in P} b_e\right)}
\end{equation*}
and recalling $g' \cdot a_e = b_e$ for $e \in P_u \cup P_w \subseteq \mathcal{C}$ and we can conclude $g' \cdot a_{(w,v)} = b_{(w,v)}$ as desired.
Therefore property~(\ref{eq:star}) extends at each step of Algorithm~\ref{alg:gauge} and the theorem is proven.
\end{proof}

Theorem~\ref{thm:gauge} has the following corollary which says that the choice of signs guaranteed by Theorem~\ref{thm:sign} is unique up to gauge transformation provided each vertex is contained in some path between boundary vertices.

\begin{cor}
If $N = (V,E)$ is a directed network embedded on a closed orientable surface with boundary such that every vertex in contained in some path between boundary vertices and there exists a collections $\{\epsilon_e\}_{e \in E}, \{\epsilon'_e\}_{e \in E} \in \{\pm 1\}^E$ such that
$$B(N,\{x_e\}_{e \in E}) = A(N, \{\epsilon_e x_e\}_{e \in E}) \mathrm{\quad and \quad} B(N,\{x_e\}_{e \in E}) = A(N, \{\epsilon'_e x_e\}_{e \in E})$$
then $\{\epsilon_e\}_{e \in E}$ and $\{\epsilon'_e\}_{e \in E}$ are gauge equivalent.
\end{cor}

\section{Acknowledgments}
The author thanks Michael Shapiro for his helpful feedback in the preparation this paper.
This work was partially supported by the National Science Foundation grant DMS-1101369.

\bibliographystyle{alpha}
\bibliography{SigningEdgesBib}

\end{document}